\documentclass[a4paper,10pt]{scrartcl}

\usepackage[fleqn]{amsmath}
\usepackage{amssymb}
\usepackage{amsthm}
\usepackage{stmaryrd}
\usepackage{bbm}
\usepackage[colorlinks, citecolor=blue]{hyperref}
\usepackage{mathtools}
\usepackage[numbers]{natbib}

\mathtoolsset{showonlyrefs}

\newcommand{\bB}{\ensuremath{\mathbb{B}}}

\newcommand{\bN}{\ensuremath{\mathbb{N}}}

\newcommand{\bP}{\ensuremath{\mathbb{P}}}

\newcommand{\bR}{\ensuremath{\mathbb{R}}}

\newcommand{\bZ}{\ensuremath{\mathbb{Z}}}

\newcommand{\ind}{\ensuremath{\mathbbm{1}}}

\newcommand{\cB}{\ensuremath{\mathcal{B}}}

\newcommand{\cF}{\ensuremath{\mathcal{F}}}

\newcommand{\cS}{\ensuremath{\mathcal{S}}}

\newcommand{\md}{\mathrm{d}}

\newcommand{\norm}[1]{\left\Vert \, #1 \, \right\Vert}

\newcommand{\ddx}[1][1]{\ifnum#1=1 \frac{d}{dx} \else \frac{d^{#1}}{dx^{#1}} \fi}
\newcommand{\ddy}[1][1]{\ifnum#1=1 \frac{d}{dy} \else \frac{d^{#1}}{dy^{#1}} \fi}
\newcommand{\ddt}[1][1]{\ifnum#1=1 \frac{d}{dt} \else \frac{d^{#1}}{dt^{#1}} \fi}

\theoremstyle{plain}
\newtheorem{thm}{Theorem}[section]  
\newtheorem{prop}[thm]{Proposition}
\newtheorem{cor}[thm]{Corollary}
\newtheorem{lem}[thm]{Lemma}
\newtheorem{defn}{Definition}[section]

\theoremstyle{definition}

\theoremstyle{remark}
\newtheorem{rem}{Remark}[section]

\numberwithin{equation}{section}

\DeclareMathOperator{\nat}{\mathbb{N}}

\newcommand{\hP}{\widehat{\bP}}

\title{One-sided continuity properties for \\the Schonmann projection} 
\author{ Stein Andreas Bethuelsen\footnote{Email: stein.bethuelsen@tum.de}  \quad and \quad  Diana Conache\footnote{Email: diana.conache@tum.de} \\
Technische Universit\"at M\"unchen}

\begin{document}
\maketitle

\begin{abstract}
We consider the plus-phase of the two-dimensional Ising model below the critical temperature. 
In $1989$ Schonmann proved that the projection of this measure onto a one-dimensional line is not a Gibbs measure. After many years of continued research which have revealed further properties of this measure, the question whether or not it is a Gibbs measure in an almost sure sense remains open.

In this paper we study  the same measure by interpreting it as a temporal process. One of our main results is that the Schonmann projection is almost surely a regular $g$-measure. That is, it does possess the corresponding \emph{one-sided} notion of almost Gibbsianness. We further deduce strong one-sided mixing properties which are of independent interest. 
Our proofs make use of classical coupling techniques and some monotonicity properties which are known to hold for one-sided, but not two-sided conditioning for FKG measures. 
\medskip

\medskip
\end{abstract}

\section{Introduction and outline of this paper}

Measures defined on an infinite product space of the form $A^{\bZ}$ are the object of study for two well-developed theories. From one point of view, measures on $A^{\bZ}$ can be regarded as so-called $g$-measures (\citet{Keane1972}), also known in the literature under various other names, such as chains with complete connections (\citet{OnicescuMihoc35}), chains of infinite order (\citet{harris1955}), uniform martingales (\citet{Kalikow1990}) or stochastic chains with unbounded memory (\citet{GALLESCO2018689}). They can be interpreted as discrete-time processes on $\bZ$ which depend on their entire history. In this sense, they can be regarded as a generalization of Markov chains. The second point of view concerns the theory of one-dimensional Gibbs measures  (\citet{Dobrushin1968, LanfordRuelle69, Georgii2011}), which refers to random fields in a spatial setting determined by a system of distributions on finite regions conditioned on given boundary conditions, called a \emph{specification}. Under sufficient regularity assumptions (\citet{FernandezMaillardCCCandGibbs2004}) the two theories coincide. However, it was recently shown (\citet{FernandezGalloMaillard2011, BissacotEndovanEnterLeNy17}) that neither one includes the other. 

In this paper, we are concerned with the \emph{Schonmann projection}, which is the projection of the plus phase of the two-dimensional Ising model in the low temperature regime onto a line, i.e.\ $\bZ\times\{0\}$, as studied in \citet{SchonmannNGibbs1989}. We will mainly be concerned with the former perspective and will regard the Schonmann projection as a temporal process. However, some comparison with the spatial fields setting will be useful. 
 
The study of projections of Gibbs measures to a sublayer has various motivations as presented in \citet{MaesRedigMoffaertNGibbs1999}. One of them is that they serve as examples of  non-Gibbsian states. An important question in the field of mathematical physics is whether the transformation of a Gibbs measure remains a Gibbs measure. As was discovered in the late 1970's this is not always the case, not even for relative simple transformations such as averaging or decimation. (See \citet{EnterFernandezSokal1993} for a thorough discussion of this phenomena). In this sense, the projection of the Ising model onto a line was introduced in \cite{SchonmannNGibbs1989} as a natural example of a transformed Gibbs measure which is non-Gibbsian, being one of the first studied models exhibiting such pathological behavior. Kozlov's characterization of Gibbsianness tells us that for a measure to be Gibbs, one needs its specification to satisfy certain non-nullness and continuity properties. The Schonmann projection does not satisfy the latter, as it was very elegantly proved in \citet{FernandezPfisterNGibbs1997}. In the context of Dobrushin's program of restoration of Gibbsianness, it was later proved that this measure is \emph{weakly Gibbs} (\citet{DobrushinShlosman1997, MaesVandeVelde1997, MaesRedigMoffaertNGibbs1999}). Several attempts were made to improve this statement, in showing that the Schonmann projection is an \emph{almost Gibbs measure}, namely that the set of continuity points of the specification has full measure. To the best of our knowledge this is still an open problem. (For a comparison between these two weaker notions of Gibbsianness see \cite{MaesRedigvonMoffaert99}).

In analogy to the description of Gibbs measures by Kozlov, a similar characterization can be given for \emph{regular} $g$-measures. We prove in this paper that the Schonmann projection is a regular $g$-measure in an almost sure sense, see Theorem \ref{thm Schonmann $g$-measure}. We do this by showing that the Schonmann projection is consistent with the $g$-functions obtained when conditioning on having ``all minuses'' or ``all pluses'' in the far past. This fact is proved by a stochastic domination argument, which can be of independent interest, see Proposition \ref{prop domi}. Moreover, we establish mixing properties, which could be interpreted as the lack of longitudinal wetting transition for our model, see Theorem \ref{thm Ising is g}. Further details on our motivation for studying the Schonmann projection from the one-sided point of the view can be found in Section \ref{sec discussion}.

\subsection*{Outline of this paper}
In the next section we  recall basic properties of the Ising model on $\bZ^2$, define the Schonmann projection and give a brief description of Gibbs measures and $g$-measures. Our main results are contained in Section \ref{sec main} together with a discussions of applications and some open questions. Proofs are postponed until Section \ref{sec proofs}.

\section{Definitions and preliminaries}
\subsection{The $2$d Ising model}\label{sec Ising}

In this subsection we give the definition of the Ising model on $\bZ^2$ and present some of its basic properties. For a more extended treatment of  the Ising model, and as a reference, see \citet[Chapter 3]{VelenikFriedliSM2017}). 

Let $\tilde{\Omega}:= \{-1,+1\}^{\bZ^2}$ and denote by $\tilde{\cF}$ the corresponding product $\sigma$-algebra. For $\Lambda \subset \bZ^2$, let $\tilde{\Omega}_{\Lambda} := \{-1,+1\}^{\Lambda}$ and let $\tilde{\cF}_{\Lambda} \subset \tilde{\cF}$ be the sub-$\sigma$-algebra which concentrates on events on $\tilde{\Omega}_{\Lambda}$. By $\omega_{\Lambda}\sigma_{\Lambda^c}$ we denote the element in $\tilde{\Omega}$ which equals $\omega_{\Lambda}$ on $\Lambda$ and $\sigma$ on $\Lambda^c$. 
Let $\tilde{\mathcal{S}}$ be the set of all finite subsets of $\mathbb{Z}^2$.
Then, for $\Lambda \in \tilde{\cS}$, the \emph{finite-volume Ising model} on $\Lambda$  with \emph{boundary condition}  $\sigma\in \tilde{\Omega}$, \emph{inverse temperature} $\beta \in (0,\infty)$ and \emph{magnetic field} $h \in \mathbb{R}$, denoted here by $\mu_{\beta,h}^{\Lambda,\sigma}$, is the probability measure on $(\tilde{\Omega}_{\Lambda}, \tilde{\cF}_{\Lambda})$ which, for $\omega_{\Lambda} \in \tilde{\Omega}_{\Lambda}$, is given by
\begin{align}
\mu_{\beta,h}^{\Lambda,\sigma}(\{\omega_{\Lambda}\}) := \frac{1}{Z_{\beta,h}^{\Lambda,\sigma}} e^{-H_{\beta,h}^{\Lambda,\sigma}(\omega_{\Lambda} )}, \quad \text{ where }Z_{\beta,h}^{\Lambda,\sigma} := \sum_{\omega_{\Lambda} \in \tilde{\Omega}_{\Lambda}} e^{-H_{\beta,h}^{\Lambda,\sigma}(\omega_{\Lambda})},
\end{align}
and where the \emph{Hamiltonian} $H_{\beta,h}^{\Lambda,\sigma} \colon \tilde{\Omega}_{\Lambda} \rightarrow \bR$ is given by
\begin{align}
H_{\beta,h}^{\Lambda,\sigma} ( \omega_{\Lambda} ) := -\beta \left[ \sum_{\{x,y\} \subset \Lambda  \atop x\sim y} \omega_x\omega_y + \sum_{x\in \Lambda, y\notin \Lambda \atop x\sim y} \omega_x \sigma_y \right]+ h \sum_{x \in \Lambda} \omega_x.
\end{align}
Here, $x \sim y$ is the notation for $x$ and $y$ being nearest-neighbors. 

The Ising model on $\bZ^2$ with parameters $(\beta,h)$ is characterized by the set of probability measures  $\mu_{\beta,h}$ on $(\tilde{\Omega},\tilde{\cF})$ satisfying the so-called DLR conditions, that is, for each $\Lambda \in \tilde{\cS}$,
\begin{align}\label{eq DLR Ising}
\mu_{\beta,h}(B) = \int \mu_{\beta,h}^{\Lambda,\sigma} (B) \mu_{\beta,h}(\md\sigma), \quad B \in  \tilde{\cF}_{\Lambda}.
\end{align}
When there is no risk of confusion, we denote by $+$ and $-$ the configurations in $\Omega$ such that $\pm_x=\pm1$ for each $x\in \bZ^2$.
It is well known that, for each $\beta \in (0,\infty)$ and $h \in \mathbb{R}$, the so-called $plus$-phase $\mu_{\beta,h}^+$ and $minus$-phase $\mu_{\beta,h}^-$, defined by
\begin{align}\label{def plus Ising}
\mu_{\beta,h}^{[-n,n]^2,\pm}(\cdot) \implies \mu_{\beta,h}^{\pm}(\cdot) \quad \text{ as } n \rightarrow \infty,
\end{align} 
are well-defined and satisfy \eqref{eq DLR Ising}. Here $\implies$ denotes weak convergence. 

By the Aizenman-Higuchi theorem, any measure satisfying \eqref{eq DLR Ising} is a convex combination of $\mu_{\beta,h}^+$ and $\mu_{\beta,h}^-$. Moreover, for $h\neq 0$ we have that $\mu_{\beta,h}^+=\mu_{\beta,h}^-$, whereas for $h=0$ there is a \emph{critical value} $\beta_c  = \ln(1+\sqrt{2})/2$ such that $\mu_{\beta,0}^+ = \mu_{\beta,0}^-$ for all $\beta \leq \beta_c$ and $\mu_{\beta,0}^+ \neq \mu_{\beta,0}^-$ for all $\beta>\beta_c$. When $\beta>\beta_c$ and $h=0$ we say that the Ising model on $\bZ^2$ is in the phase transition regime,  and when $h\neq0$ or when $h=0$ and $\beta<\beta_c$ it is in the uniqueness regime. The Ising model with $\beta=\beta_c$ and $h=0$ is the critical case.

\subsection{The Schonmann projection}

We start by introducing some more notation. In the sequel,  $A$ is a finite state space (or alphabet)  which we equip with the discrete $\sigma$-algebra $\mathcal{A}$ formed of all subsets of $A$. 
Further, we consider the configuration space $\Omega=A^{\bZ}$ endowed with the product topology and with the corresponding product $\sigma$-algebra $\cF$. Similarly, for $n\in \bZ$, we consider $\Omega_{>n}=A^{\bZ_{n,+}}$ and $\Omega_{<n}=A^{\bZ_{n,-}}$ equipped with $\sigma$-algebras $\cF_{>n}$ and $\cF_{<n}$, where $\bZ_{n,+} = \bZ \cap [n, \infty)$ and $\bZ_{n,-} = \bZ \cap (-\infty, n]$, respectively.  
Given $\Lambda\subset\mathbb{Z}$, we denote by $\omega_{\Lambda}$ the restriction of a configuration $\omega\in\Omega$ to the volume $\Lambda$  
and denote the corresponding configuration space by $\Omega_{\Lambda}$.  
An important notation which will be used throughout this paper for elements in $\Omega_{[i,j]}$, $i,j\in\bZ\cup\{\pm\infty\}$ with $i\leq j$, is $\omega_{i}^{j}=\omega_i\omega_{i+1}\ldots\omega_{j}$. 
Moreover, we write $\mathcal{S}$ for the set of all finite subsets of $\mathbb{Z}$.

The focus of this paper is the projection of the plus-phase of the Ising model onto $\bZ \times \{0\}$, also referred to as the Schonmann projection.  
That is, by taking $A=\{-1,+1\}$, we consider the probability measure $\nu_{\beta,h}^{+}$ on $(\Omega,\cF)$ which is characterized by
\begin{equation}
\nu_{\beta,h}^{+}(\{ \omega \colon \omega_{\Lambda} = \eta_{\Lambda}\}) :=  \mu_{\beta,h}^{+}(\{\tilde{\omega} \colon  \tilde{\omega}_{\Lambda \times \{0\}} = \eta_{\Lambda} \}) 
\end{equation} 
for each $\Lambda \subset \cS$ and $\eta \in \Omega$, and where $\mu_{\beta,h}^{+}$ was introduced in \eqref{def plus Ising}. Similarly, the projection of $\mu_{\beta,h}^-$ onto $\bZ\times\{0\}$ is denoted by $\nu_{\beta,h}^-$. One important property of these measures which we will use several times in the sequel is that they are \emph{translation-invariant}, that is,
\begin{equation}
\nu_{\beta,h}^{\pm}(\{ \omega \colon \omega_{\Lambda} = \eta_{\Lambda}\})
= \nu_{\beta,h}^{\pm}(\{ \omega \colon \omega_{T^{-1}\Lambda} = \eta_{T^{-1}\Lambda}\})
\end{equation}
for all $\Lambda \in \cS$, where $T^{-1}\Lambda:= \{ x \in \bZ \colon x+1 \in \Lambda \}$.

\subsection{Two-sided conditioning. The class of 1-dimensional DLR Gibbs measures}

In the so-called DLR (Dobrushin, Lanford and Ruelle) approach of Mathematical Statistical Mechanics, Gibbs measures at infinite volume are probability measures defined by a system of conditional probabilities, conditioned on configurations outside finite sets $\Lambda$, called \emph{specification}, in the same manner as the definition of the Ising model via a Hamiltonian as given in Section \ref{sec Ising}.  We refer to \cite[Chaper 6]{VelenikFriedliSM2017} for precise definitions of this program. 

Instead we present here an equivalent characterization of \emph{translation-invariant} Gibbs measures which is closer to our description of $g$-measures in the following sections. This characterization was given and proven in \citet{BerghoutFernandezVerbitskiy2017} and is the content of the following proposition.

\begin{prop}[Theorem 3.1 in \cite{BerghoutFernandezVerbitskiy2017}]\label{prop equiv characterization Gibbs}
Let $\mu$ be a translation-invariant measure on $\Omega$ with underlying process 
$(X_i)_{i\in\bZ}$. Then $\mu$ is a Gibbs measure if and only if it there is a continuous function  $\gamma:\Omega\to (0,1)$ satisfying
\begin{equation}\label{eq Gibbs normalization}
\sum_{\omega_0\in A} \gamma(\omega_{-\infty}^{-1}, \omega_0,\omega_{1}^{\infty})=1, \textrm{ for all }\omega\in\Omega
\end{equation}
and for which we have that
\begin{align}\label{eq: equiv characterization Gibbs}
\mu(X_0=\omega_0\vert X_{-\infty}^{-1}=\omega_{-\infty}^{-1}, X_{1}^{\infty}=\omega_{1}^{\infty})=\gamma(\omega_{-\infty}^{-1}, \omega_0,\omega_{1}^{\infty}), 
\end{align}
for $\mu$-a.a. $\omega\in\Omega$.
\end{prop}

Following Proposition \ref{prop equiv characterization Gibbs}, we next introduce the concept of \emph{almost Gibbsian} measures.

\begin{defn}
Let $\mu$ be a translation-invariant measure on $\Omega$ with underlying process 
$(X_i)_{i\in\bZ}$. Then $\mu$ is \emph{almost Gibbs} if and only if there is a function  $\gamma:\Omega\to (0,1)$ satisfying \eqref{eq Gibbs normalization} which is continuous for $\mu$-a.a.\ $\omega\in\Omega$, and for which we have that \eqref{eq: equiv characterization Gibbs} is satisfied for $\mu$-a.a.\ $\omega\in\Omega$.
\end{defn}

\subsection{One-sided conditioning. The class of $g$-measures}

We think of the class of $g$-measures as a general description of stochastic processes with long-range memory. Similarly to Gibbs measures, which can be seen as a generalization of Markov random fields, these processes can be seen as an analogous generalization of Markov chains. We shall restrict to such processes on $\Omega$ which are translation-invariant.

\begin{defn}
Let $\mu$ be a translation-invariant measure on $\Omega$ with underlying process 
$(X_i)_{i\in\bZ}$. Then $\mu$ is a \emph{regular $g$-measure} if and only if there is a continuous function $P \colon A \times\Omega_{<0}\to (0,1)$ satisfying
\begin{equation}\label{eq normalization}
\sum_{\omega_0\in A}P\bigl(\omega_0 \bigm| \omega_{-\infty}^{-1}\bigr)\;=\;1
\quad\forall\omega_{-\infty}^{-1}\in\Omega_{<0}.
\end{equation}
and for which we have that
\begin{equation}\label{eq consistency}
\mu\bigl(X_{0}=\omega_{0} \bigm| X_{-\infty}^{-1}=\omega_{-\infty}^{-1}\bigr)
\;=\;P\bigl(\omega_{0} \bigm| \omega_{-\infty}^{-1}\bigr)
\end{equation}
for all $\omega_{0}\in A$ and $\mu$-a.e.\ $\omega_{-\infty}^{-1}\in\Omega_{<0}$. 
\end{defn}

We remark here the similarity of this definition and the equivalent characterization of Gibbs measures stated in Proposition \ref{prop equiv characterization Gibbs}. In some sense, one could say that the function $\gamma$ appearing in Equation \refeq{eq: equiv characterization Gibbs} is \emph{a two-sided $g$-function}.

In more generality, a $g$-measure is a measure $\mu$ on $\Omega$ satisfying \eqref{eq consistency} with respect to a \emph{ $g$-function} $P \colon A \times\Omega_{<0}\to [0,1]$ satisfying \eqref{eq normalization}. Such a $g$-function is said to be \emph{regular} if it is continuous and maps all its elements to $(0,1)$. To be more precise, the  $g$-function $P$ is continuous if
\begin{equation}
var_k(P) := \sup_{\omega_{-\infty}^{-1}, \sigma_{-\infty}^{-1} \in \Omega<0,  \atop \omega_{-k}^{-1} = \sigma_{-k}^{-1}}  \sup_{a \in A} | P(a \mid \omega_{-1}^{-\infty}) - P(a \mid \sigma_{-1}^{-\infty}) | \rightarrow 0 \text{ as } k \rightarrow \infty.
\end{equation}

Similar to the notion of almost Gibbs, we now introduce almost regular $g$-measures:

\begin{defn}
Let $\mu$ be a translation-invariant measure on $\Omega$ with underlying process 
$(X_i)_{i\in\bZ}$. Then $\mu$ is an \emph{almost regular $g$-measure} if and only if there is a function $P \colon A \times\Omega_{<0}\to (0,1)$ satisfying \eqref{eq normalization} which is continuous $\mu$-a.s.\ and for which $\mu$ is a $g$-measure.
\end{defn}

\begin{rem}
Most of the literature on $g$-measures assumes a priori regularity of the $g$-function, although this assumption has been relaxed in some recent publications, see e.g.\  \citet{GalloPaccaut2013}.
\end{rem}

\section{Main results and discussion}\label{sec main}

\subsection{Main results}

In this subsection we consider the measure $\nu_{\beta,h}^+$ from the one-sided point of view in the sense of $g$-measures. 
We first consider the regime of parameters in which the two-dimensional Ising model is unique, for which we obtain analogous statements as for the two-sided case.

\begin{prop}\label{prop unique}
In the uniqueness regime $\nu_{\beta,h}^+$  is a regular $g$-measure for a certain $g$-function $P$. Moreover,
 $\sum_{k \geq 1} var_k(P) < \infty,$ 
and hence $\nu_{\beta,h}^+$  is the unique $g$-measure consistent with $P$.
\end{prop}

Proposition \ref{prop unique} is reminiscent of the fact that the Ising model on $\bZ^2$, in the entire uniqueness regime, belongs to the class of Gibbs models having \emph{complete analyticity}, as proven in \citet{SchonmannSchlosman1995}. Consequently, as follows by standard arguments, in the uniqueness phase, the measure $\nu_{\beta,h}^{+}$ can be considered as both a Gibbs measure and a $g$-measure. (The former was proven in \citet{LorincziProjectedGibbs1995}, see also \cite[Theorem 4.4]{MaesRedigMoffaertNGibbs1999}). As \cite{SchonmannNGibbs1989} proved,  this is no longer the case in the phase transition regime for which $\nu_{\beta,h}^+$ is not a Gibbs measure. 
From this result more can be said about the interpretation of  $\nu_{\beta,h}^{+}$ as a $g$-measure. In particular, it  cannot satisfy such strong continuity properties as in Proposition \ref{prop unique}, as we state next.

\begin{prop}\label{prop not summable}
Let $h=0$, $\beta>\beta_c$ and assume there is a $g$-function $P$ such that $\nu_{\beta,h}^+$ is a $g$-measure consistent with it. Then
 $\sum_{k \geq 1} var_k(P) = \infty.$ 
\end{prop}
Proposition \ref{prop not summable} follows by combining the results in
\cite{FernandezMaillardCCCandGibbs2004} on equivalence of $g$-measures and Gibbs measures with the fact that $\nu_{\beta,h}^+$ is not a Gibbs measure in the phase transition regime. An important step in the proof in \cite{SchonmannNGibbs1989} of that $\nu_{\beta,h}^+$ is not a Gibbs measure in the phase transition regime, see \cite[Lemma 1]{SchonmannNGibbs1989}, is the following statement: 

for $\beta>\beta_c$, and for each $n \in \bN$ there exists $N(n)\geq n$ such that
\begin{equation}\label{eq two sided nomixing}
\lim_{n \rightarrow \infty} \mu_{\beta,0}^+( \cdot \mid \eta = -1 \text{ on } [(-N(n),-n)\cup (n,N(n)] \times \{0\} ) = \mu_{\beta,0}^-(\cdot).
\end{equation}
We now state our main result, which shows that the one-sided version of Equation \eqref{eq two sided nomixing} exhibits a very different behavior.

\begin{thm}\label{thm Schonmann $g$-measure}
Let $h =0$ and $\beta \geq \beta_c$.
\begin{description}
\item[a)] There exists a $g$-function $P$ to which $\nu_{\beta,0}^+$ is the \textbf{unique} consistent measure. Furthermore, $P$  is \textbf{almost regular} with respect to $\nu_{\beta,0}^+$.
 \item[b)] If $\beta>\beta_c$ we have that $\nu_{\beta,0}^+$ is \textbf{$\phi$-mixing}, that is,
\begin{align}\label{eq first wetting}
\lim_{n \rightarrow \infty} \sup_{B \in \cF_{> n}} \sup_{A \in \cF_{<0}} | \nu_{\beta,0}^+(B \mid A) - \nu_{\beta,0}^+(B)| =0.
\end{align}
\end{description}
\end{thm}
The one-sided mixing property in Equation \eqref{eq first wetting} should be compared with the two-sided property given in Equation \eqref{eq two sided nomixing} (see also Theorem \ref{thm Ising is g} below). Further, an immediate consequence of the $\phi$-mixing property, by applying \cite[Theorem 1]{GALLESCO2018689}, is the following:
\begin{cor}
Let $h =0$ and $\beta > \beta_c$ and let $P$ be as in Theorem \ref{thm Schonmann $g$-measure} a). Then, for every $\sigma_{-\infty}^{-1},\tau_{-\infty}^{-1}\in\Omega_{<0}$, we have that
 \begin{align}\label{eq weak square summable}
\sum_{k =0}^{\infty} \sum_{a=\pm 1} \bigg|P\Big(a\big|\sigma_{-\infty}^{-(k+1)}\omega_{-k}^{-1}\Big)-P\Big(a\big|\tau_{-\infty}^{-(k+1)}\omega_{-k}^{-1}\Big)\bigg|^2< \infty
\end{align}
for $\nu_{\beta,0}^+\text{-a.e. } \omega_{-\infty}^{-1} \in \Omega_{<0}$.
\end{cor}

We leave the question whether $\nu_{\beta,0}^{+}$ is a regular $g$-measure open, but note that, as seen by Proposition \ref{prop not summable}, in any case it has relatively strong dependence on its past. Interestingly, by adapting the methods in \cite{MaesRedigMoffaertNGibbs1999} and \citet{LorincziVelde1994}, 
we do obtain strong regularity properties for certain decimations of $\nu_{\beta,0}^+$. To be more precise, for $l,k \in \bN$, let $\nu_{l,k}^+$  be the projection of $\nu_{\beta,h}^+$ onto  $\bZ_{l,k}:=\{x \in \bZ \colon x \mod k+l = i, i =0,\dots, l-1\} \subset \bZ$ consisting of all translations of the interval $[0,l-1]$ which are separated by gaps of length $k$.

\begin{prop}\label{prop decimation}
For any $l \in \bN$ there is a $K=K(l)$ such that,  for all $k\geq K$ and all $\beta \geq \beta(l,k)$ large, the measure $\nu_{l,k}^+$ is a regular $g$-measure with summable variation. 
\end{prop}

\begin{rem}
 \citet{LorincziVelde1994} showed that, for large enough $\beta>\beta_c$, the measure $\nu_{1,k}$, $k \geq 2$, is a Gibbs measure, and their approach can also be used to show that it is a $g$-measure. Interestingly, in the same paper, they conjectured that $\nu_{1,1}$ is, similar to the Schonmann projection $\nu_{1,0}$, not a Gibbs measure. On the other hand, from our results it follows that both of them are almost regular $g$-measures.
\end{rem}

We end this section with a strengthening of Theorem \ref{thm Schonmann $g$-measure}b)  for the Ising model on $\bZ^2$.  For this, let $C_{\theta} = \{ (x,y) \in \bZ^2 \colon x\geq 0, |y| \leq  e^{\theta  |x|}\}$, where $\theta \in (0,\infty)$, and let, for $n \in \bN$, $C_{\theta,n} := C_{\theta} \cap [n,\infty) \times \bZ$. Moreover, let $\tilde{\cF}_{<0}$ be the $\sigma$-algebra concentrating on the events on $(-\infty,0)\times\{0\}$.

\begin{thm}\label{thm Ising is g}
Consider the Ising model on $\bZ^2$ in the phase transition regime. Then there exists $\theta\in (0,\infty)$, depending on $\beta$, such  that
\begin{align}\label{eq mixing Ising}
\lim_{n \rightarrow \infty} \sup_{B \in \tilde{\cF}_{_{\theta,n}}} \sup_{A \in \tilde{\cF}_{<0}} | \mu_{\beta,0}^+(B \mid A) - \mu_{\beta,0}^+(B)| =0.
\end{align}
\end{thm}

The one-sided mixing property in the above theorem can be interpreted in the sense of \emph{wetting}. The wetting phenomenon appears in systems with several types of particles inside of a box, whose physical parameters are such that they allow coexistence of phases. If, for example, one looks at the Ising model prepared in the plus phase, but the horizontal bottom wall of the box preferentially adsorbs the minus phase, one notices the appearance of a thin film of the minus phase between the wall and the bulk plus phase. A good reference on the topic are the notes of Pfister and Velenik, see \cite{PfisterVelenik1998}. 

In our case, Theorem \ref{thm Ising is g} shows that a piece of wall adsorbing the minus phase generates no longitudinal wetting. Moreover, our result should be contrasted with the recent results on one-sided wetting and discontinuity for the Dyson-Ising model, see \cite[Proposition 4]{BissacotEndovanEnterLeNy17}.

\begin{rem}
As $\beta$ increases, the parameter $\theta$ in Theorem \ref{thm Ising is g} goes to infinity and  $C_{\theta}$ converges to the entire half plane $\{(x,y)\in\bZ^2 : x\geq 0\}$.
\end{rem}

\subsection{Discussion and open questions}\label{sec discussion}
In this subsection we discuss several applications of the theorems presented so far, as well as some (in our opinion) intriguing open questions.
\begin{enumerate}
\item \citet{BergSteifFC1999} proved that the plus-phase of the Ising model with $h=0$ can be \emph{perfectly sampled} if and only if $\beta<\beta_c$. 
More recently, \citet{GalloTakahaskiAttractive2014} showed that an attractive  and regular $g$-measure has a perfect sampling algorithm if and only if it is unique. By standard arguments in this paper (see Lemmas \ref{lem ergodic} and \ref{lem uniqueness}), the projection of the Ising model onto a line, $\nu_{\beta,h}^+$, is consistent with an attractive $g$-function. Furthermore, by Theorem \ref{thm Schonmann $g$-measure} and Proposition \ref{prop unique} it is unique. Propositions \ref{prop unique} and \ref{prop decimation} thus, by the results in \cite{GalloTakahaskiAttractive2014}, give us conditions for when $\nu_{\beta,h}^+$, or decimations thereof, can be perfectly sampled. It is an intriguing open question whether $\nu_{\beta,h}^+$ can be perfectly sampled in the phase transition regime. We plan to address this question in future work.

\item All theorems and propositions of this section can be extended to the projection of the two-dimensional Ising model onto $\bZ \times \{0,\dots,k\}$, $k \in \bN$. However, it is important for several of our arguments  that the projections under consideration are translation-invariant. It would be interesting to extend our results to also include projections which are not translation-invariant. Moreover, the proof techniques used in this paper to study the $2$-dimensional Ising model can presumably also be extended to higher dimensions and to models sharing similar monotonicity properties.

\item The mixing property in Equation \eqref{eq first wetting}  is an essential strengthening of \citet[Lemma 2]{SchonmannLDPising1987} and a rather powerful mixing property. For instance, general concentration inequalities follow by the work of \citet{Samson2000}.

\item In the two-sided case, the lack of Gibbsianness follows as a consequence of a wetting transition (see \cite{EnterFernandezSokal1993}, \cite{FernandezPfisterNGibbs1997}, \cite{SchonmannNGibbs1989}). By Theorem \ref{thm Schonmann $g$-measure}b) and Theorem \ref{thm Ising is g} this is in contrast to the one-sided case, which strongly suggests that the Schonmann projection is indeed a \emph{regular $g$-measure}.

\item  We hope that our result showing that the Schonmann projection is an almost regular $g$-measure can cast some light on, and possible be of importance to, the question whether it is almost Gibbsian or not. \cite{FernandezMaillardCCCandGibbs2004} provides sufficient regularity assumptions for the theory of $1$d Gibbs measures and $g$-measures to be equivalent. Is there an analogous equivalence between almost regular $g$-measures and almost Gibbsian measures which requires weaker assumptions?

\item 
A note on the critical case $h=0$ and $\beta=\beta_c$ is in place. In this case, as follows by \cite{FernandezPfisterNGibbs1997}, the measure $\nu_{\beta_c,0}^+$ is almost Gibbs, and by Theorem \ref{thm Schonmann $g$-measure} it is also an almost regular $g$-measure. It is an interesting question whether or not $\nu_{\beta_c,0}^+$ is in fact a Gibbs measure or a regular $g$-measure.

\end{enumerate}

\section{Proofs}\label{sec proofs}

\subsection{Proof of Theorem \ref{thm Schonmann $g$-measure}a)}
Let $(X,\cB)$ stand for $(\Omega,\cF)$ or $(\tilde{\Omega},\tilde{\cF})$.  
Associate to $X$ the partial ordering such that $\xi \leq \eta$  if and only if $\xi(x) \leq \eta(x)$ for all $x \in \bZ$, or $\bZ^2$, respectively. An event $B \in \cB$ is said to be \emph{increasing} if $\xi\leq\eta$ implies $1_{B}(\xi)\leq 1_{B}(\eta)$. More generally, a function $f\colon X \rightarrow \bR$ is increasing if $\xi\leq\eta$ implies $f(\xi)\leq f(\eta)$. 
A  $g$-function $P \colon \{-,+\}\times \Omega_{-\infty}^{-1} \rightarrow [0,1]$ for which $P(+\mid\cdot)$ is increasing is said to be \emph{attractive}.

\begin{defn}\label{def FKG}
Let $\mu$ be a probability measure on $(X,\cB)$. We say that
\begin{enumerate}
\item $\mu$ is FKG (or positively associated) if  $\mu(B_1 \cap B_2) \geq \mu(B_1) \mu(B_2)$ for any two increasing events $B_1,B_2 \in \cB$.
\item $\mu$ is strong FKG (or lattice FKG)  if for every $\Lambda$ finite and $\sigma \in X$, the measure $\mu(\cdot \mid \eta = \sigma \text{ on } \Lambda )$ is positively associated, where $\{\eta = \sigma \text{ on } \Lambda\} := \{ \eta \in X \colon \eta_x = \sigma_x \: \forall \: x\in \Lambda\}$.
\end{enumerate}
\end{defn}
The Ising model $\mu_{\beta,h}^+$ is known to be strong FKG and this property directly transfers to the measure $\nu_{\beta,h}^+$. By using this property,  via a global specification argument as thoroughly described in \cite{FernandezPfisterNGibbs1997}, it can be shown that the functions $P_{\beta,h}^{+,\pm} \colon \{-,+\}\times \Omega_{-\infty}^{-1} \rightarrow (0,1)$ given by
\begin{equation}\label{def particular $g$-functions}
P_{\beta,h}^{+,\pm}(\pm \mid \omega^{-1}_{-\infty}) := \lim_{n \rightarrow \infty}\lim_{l\rightarrow \infty} \nu_{\beta,h}^+ \left( \eta_0=\pm  \mid \eta^{-1}_{-(n+l)} = \pm^{-(n+1)}_{-(n+l)}\omega^{-1}_{-n} \right) 
\end{equation}
are well-defined.

\begin{lem}\label{lem ergodic}
For every $\beta \in (0,\infty)$ and $h \in \mathbb{R}$, the functions $P_{\beta,h}^{+,\pm}$ are well-defined and attractive $g$-functions, and they  satisfy $P_{\beta,h}^{+,-}(+\mid \cdot)\leq P_{\beta,h}^{+,+}(+\mid \cdot)$. Moreover, the following probability measures 
\begin{align}
&\nu_{\beta,h}^{+,\pm} (\cdot) := \lim_{n \rightarrow \infty} \lim_{l\rightarrow \infty}  \nu_{\beta,h}^+ \left( \cdot \mid \eta^{-n}_{-(n+l)} = \pm^{-n}_{-(n+l)} \right)
\end{align}
are consistent and extremal with respect to $P_{\beta,h}^{+,+}$ and $P_{\beta,h}^{+,-}$ respectively.
\end{lem}
\begin{proof}
A proof of this statement is given in \cite{HulseAttractive1991}; see Lemma 2.3 and Lemma 3.1 therein.
\end{proof}

\begin{lem}\label{lem uniqueness}
Let $\beta \in (0,\infty)$ and $h\in\mathbb{R}$, then 
\begin{description}
\item[a)]  $\nu_{\beta,h}^{+}$ is consistent with $P_{\beta,h}^{+,+}$.
\item[b)] if $P_{\beta,h}^{+,+}=P_{\beta,h}^{+,-}$, then $\nu_{\beta,h}^{+}$ is a regular $g$-measure.
\item[c)] if $\nu_{\beta,h}^+$ is consistent with $P_{\beta,h}^{+,-}$, then $\nu_{\beta,h}^{+}$ is an almost regular $g$-measure.
\item[d)]  if $\nu_{\beta,h}^{+} = \nu_{\beta,h}^{+,-}$, then $\nu_{\beta,h}^{+}$  is the unique measure consistent with $P_{\beta,h}^{+,+}$.
\end{description} 
\end{lem}

\begin{proof}
From the definition of $\mu_{\beta,h}^+$ it directly follows (see Equation \eqref{def plus Ising}) that $\nu_{\beta,h}^+=\nu_{\beta,h}^{+,+}$, which we know, by the previous lemma, is consistent with $P_{\beta,h}^{+,+}$. This proves a).
Hence, that $\nu_{\beta,h}^{+}$ is a regular $g$-measure if $P_{\beta,h}^{+,+}=P_{\beta,h}^{+,-}$ follows merely by definition thus yielding b). 
Similarly, if $\nu_{\beta,h}^+$ is consistent with $P_{\beta,h}^{+,-}$, then,  since $P_{\beta,h}^{+,-}(+\mid\cdot) \leq P_{\beta,h}^{+,+}(+\mid\cdot)$ and since $\nu_{\beta}^+(P_{\beta,h}^{+,+} (+\mid~\cdot)- P_{\beta,h}^{+,-}(+\mid\cdot))=0$, we have that $P_{\beta,h}^{+,+}(\pm\mid\omega^{-1}_{-\infty})=P_{\beta,h}^{+,-}(\pm\mid\omega^{-1}_{-\infty})$ for $\nu_{\beta}^+$-a.e. $\omega^{-1}_{-\infty} \in \Omega^{-1}_{-\infty}$. That is,  $\nu_{\beta,h}^{+}$ is an almost regular $g$-measure, therefore proving c).
For d), if furthermore $\nu_{\beta,h}^{+,-} = \nu_{\beta,h}^+$, then $\nu_{\beta}^+$ is the unique measure consistent with $P_{\beta,h}^{+,+}$ since, by the strong FKG property,   any measure $\nu$ consistent with $P_{\beta,h}^{+,+}$ satisfies $\nu_{\beta,h}^{+,-} \leq \nu \leq  \nu_{\beta,h}^{+,+} $.
\end{proof}

In order to conclude Theorem \ref{thm Schonmann $g$-measure}, we will show that conditions c) and d) in Lemma \ref{lem uniqueness} indeed hold. 
To achieve this, we use a domination argument which we think is of independent interest. For this, let $\pi_{\rho}$ denote the product Bernoulli-$\{-1,+1\}$ measure on $(\Omega,\cF)$ where $\pi_{\rho}(\eta_0=+) =\rho \in [0,1]$. We say that a measure $\mu$ on $(\Omega,\cF)$ \emph{stochastically dominates} another measure $\nu$  if $\nu(B) \leq \mu(B)$ for all increasing events $B \in \mathcal{F}$.

\begin{prop}\label{prop domi}
For each $\beta>\beta_c$ there exists a $\rho \in (0,1)$ such that $\nu_{\beta,0}^+$ stochastically dominates $\pi_{\rho}$, but $\pi_{\rho}$ is not stochastically dominated by  $\nu_{\beta,0}^-$.
\end{prop}
The proof of Proposition \ref{prop domi} is deferred to the next subsection. 
Note that its statement contrasts \citet[Proposition 1.2]{LiggettSteifSD2006} which says that the analogous statement is not true for $\mu_{\beta,h}^+$ and $\mu_{\beta,h}^-$  (see also \citet{WarfheimerSDFuzzy2010} for further results in this direction). Based on Proposition \ref{prop domi}, an immediate application of
\cite[Theorem 1.2]{LiggettSteifSD2006} is the following:

\begin{cor}\label{cor domi}
Let $\beta \in (0,\infty)$ and $h\in\mathbb{R}$. Then both $\nu_{\beta,h}^{+,\pm}$ stochastically dominate $\pi_{\rho_{\max}^+}$, where 
$\rho_{\max}^{+} :=  \sup \{ \rho \in [0,1] \colon \pi_{\rho} \leq \nu_{\beta,h}^{+} \}$.
\end{cor}

\begin{proof}[Proof of Theorem \ref{thm Schonmann $g$-measure}a)]
Let $\beta\in (0,\infty)$ and $h\in \mathbb{R}$, and define
\begin{align}\label{eq limit measure}
\mu(\cdot):=\lim_{n \rightarrow \infty} \mu_{\beta,h}^+( \cdot \mid \eta = -1 \text{ on } (-\infty,-n] \times\{0\} ).
\end{align}
By the strong FKG property it follows that this limit exists. 
Furthermore, the measure $\mu$ is consistent with respect to the Ising model on $\bZ^2$ with parameters $\beta$ and $h$. Consequently, if $h\neq 0$ or $\beta\leq\beta_c$, we conclude that  $\mu = \mu_{\beta,h}^+$. In particular, we have that $\nu_{\beta_c,0}^{+,-}=\nu_{\beta_c,0}^{+,+}$.

Similarly, for $h=0$ and $\beta>\beta_c$, by Corollary \ref{cor domi}, the projection of $\mu$ onto  $\bZ\times \{0\}$ stochastically dominates $\pi_{\rho_{\max}^{+}}$. Since, by the Aizenman-Higuchi theorem, we have that $\mu= \alpha\mu_{\beta}^+ + (1-\alpha)\mu_{\beta}^-$ for some $\alpha \in [0,1]$, it thus follows, by applying Proposition \ref{prop domi}, that $\mu = \mu_{\beta,h}^+$. In particular, we have that $\nu_{\beta,0}^{+,-}=\nu_{\beta,0}^{+,+}$ for all $\beta\geq\beta_c$. Hence, by Lemma \ref{lem ergodic}c) and d), we conclude the proof of Theorem \ref{thm Schonmann $g$-measure}a).
\end{proof}

\subsection{Proof of Proposition \ref{prop domi}}

For $\beta$ large, Proposition \ref{prop domi} follows by a simple application of Peierls argument. For instance, by a slight modification of the proof in \cite{VelenikFriedliSM2017} (see page $112$-$114$ therein), one can show that, for $\beta>\ln3/2$,
\begin{align}
\nu_{\beta,0}^+( \eta = -1 \text{ on } [-n,-1]) \leq 1/2^n, \quad \text{ for all } n \text{ large}.
\end{align}
Consequently, by \cite[Theorem 1.2 ]{LiggettSteifSD2006}, for such $\beta$, the measure $\nu_{\beta,0}^+$ stochastically dominates $\pi_{1/2}$ which $\nu_{\beta,0}^-$ clearly cannot dominate (indeed, $\nu_{\beta,0}^-(\eta(0)=1)<1/2$ for all $\beta>\beta_c$).

We now continue with a proof of Proposition \ref{prop domi} which applies to all $\beta>\beta_c$. For this, we first briefly recall the random cluster representation of the Ising model and the so-called Edwards-Sokal coupling, following \citet[Chapter 6]{GeorgiiHaggstromMaes2001}:

Let $\bB$ be the set of all nearest-neighbor bonds in $\bZ^2$ and, for $\Lambda \subset \bZ^2$, let $\bB_{\Lambda}$ be the set of all edges of $\bB$ with at least one endpoint in $\Lambda$. Given $\eta\in \bB$ we say that $\Lambda \subset \bZ^2$ forms an $\eta$-open cluster if $\eta(e)=1$ for all $e \in \bB$ contained in $\Lambda$ and $\eta(e)=0$ for all $e=(x,y) \in \bB$ with $x\in \Lambda$, $y \in \Lambda^c$. Next, for $n\in \bN$, denote by $\phi_{p,n}^1$ the probability measure on $\{0,1\}^{\bB}$ in which each $\eta \in \{0,1\}^{\bB}$ is assigned a probability proportional to 
\begin{align}
\ind_{\{\eta = 1 \text{ off } \bB_{ [-n,n]^2}\}}\bigg\{ \prod_{e \in \bB_{[-n,n]^2}} p^{\eta(e)}(1-p)^{1-\eta(e)} \bigg\} 2^{k(\eta,n)+1},
\end{align}
where $k(\eta,n)$ is the number of all $\eta$-open clusters entirely contained in $[-n,n]^2$.  The measure $\phi_{p,n}^1$ is the random-cluster distribution in $[-n,n]^2$ with parameters $p$ and $2$ and wired boundary condition.
 
 Next, for $\beta>0$, let $\bP_{\beta,0}^{n,+}$ be the probability measure on $\{-1,+1\}^{ \Lambda} \times \{0,1\}^{\bB}$ corresponding to picking a random element of $\{-1,+1\}^{\Lambda} \times \{0,1\}^{\bB}$ according to the following two-step procedure:
\begin{enumerate}
\item Assign to each vertex of $\bZ^2\setminus[-n,n]^2$ value $1$, and to all edges of $\bB \setminus \bB_{[-n,n]^2}$ value $1$.
\item Assign to each vertex in $[-n,n]^2$ a spin value chosen from $\{-1,+1\}$ according to the uniform distribution, assign to each edge in $\bB_{[-n,n]^2}$ value $1$ or $0$ with respective probabilities $p$ and $ 1-p$, with $p = 1-e^{-2\beta}$, and do this independently for all vertices and edges.
\item Condition on the event that no two vertices with different spins have an edge with value $1$ connecting them. 
\end{enumerate}
The measure $\bP_{\beta,0}^{n,+}$ is known as the \emph{Edwards-Sokal coupling} of the random cluster measure on $[-n,n]^2$ and the Ising model. In particular, we have that the vertex and edge marginals of $\bP_{\beta,0}^{n,+}$ are $\mu_{\beta,0}^{[-n,n]^2,+}$ and $\phi_{p,n}^1$ respectively. We denote by $\bP_{\beta,0}^+$ the limiting measure of $\bP_{\beta,0}^{+,n}$ as $n\rightarrow \infty$. 

\begin{lem}\label{lem exponential decay}
For $\Lambda \subset \bZ^2$, write $\{ \Lambda \leftrightarrow \infty\}$ for the event \[\big\{\eta \in \{0,1\}^{\bB}\colon \text{ there is an } \infty \: \eta\text{-open cluster intersecting } \Lambda\big\}.\] 
Let $\beta>\beta_c$. Then there exist constants $C,c>0$ such that 
\begin{align}
\bP_{\beta,0}^{+} \Big( \big\{[-n,-1] \leftrightarrow \infty\big\}^c \Big) \leq Ce^{-cn}, \quad \forall \: n \in \bN.
\end{align}
\end{lem}

\begin{proof}
The statement follows by adapting the duality arguments presented in Chayes et al. \cite[Pages 437-440]{ChayesChayesSchonmann1987}. One makes use of the fact that if a given configuration $\eta$ is distributed according to the random cluster distribution $\phi_{p,n}^{1}$ with wired boundary conditions, then the dual configuration $\eta^{*}$ is also distributed according to a  "dual" random cluster measure $\Phi_{p^*,n}^{1,f}$   with free boundary conditions. Moreover, $p^*=\frac{2-2p}{2-p}$ and the inverse temperature of the dual system satisfies the Kramers-Wannier dual relation $e^{-\beta^*}=\tanh(\frac{1}{2}\beta)$. The most important remark here is that for $\beta>\beta_c$, one has that $\beta^*<\beta_c$.

Denoting by $\mathbb{P}_{\beta^*,0}^{*}$ the Edwards-Sokal coupling of the dual Ising model at inverse temperature $\beta^*$ with $\Phi_{p^*,n}^{1,f}$, we have that 
\begin{equation}
\mathbb{P}_{\beta,0}^{+}(\{[-n,-1]\leftrightarrow \infty\}^c)=\mathbb{P}_{\beta^*,0}^{*}(ODC[-n,-1]),
\end{equation}
where $ODC[1,n]$ denotes the event that there exists an open dual cluster surrounding $[1,n]$. This observation together with the exponential decay of correlations in the dual Ising model at inverse temperature $\beta^*$ (see discussion on page 439 in \cite{ChayesChayesSchonmann1987}) yield the desired conclusion. 
\end{proof}

\begin{proof}[Proof of Proposition \ref{prop domi}]
Recall that $\rho_{\max}^{\pm} = \sup \{ \rho \in [0,1] \colon \pi_{\rho} \leq \nu_{\beta}^{\pm} \}$.
By  \cite[Theorem 1.2 ]{LiggettSteifSD2006} and symmetry of the model, we have $\rho_{\max}^{\pm} = 1-e^{-\theta_{\pm}}$, where
\begin{align}\label{eq ldp limit}
\theta_{\pm}:= - \lim_{n \rightarrow \infty}n^{-1} \log \nu_{\beta,0}^{+} ( \eta = \mp1 \text{ on } [-n,-1] ),
\end{align}
and where the existence of the limit in \eqref{eq ldp limit} follows by standard sub-additivity arguments. For the proof of the proposition it is thus sufficient to show that $\theta_{-} < \theta_{+}$ for all $\beta>\beta_c$.

For $\beta$ large, the fact that $\theta_{+} >\theta_{-}$ follows easily by means of a Peierls argument. To show that this holds for all $\beta >\beta_c$, we make use of the Edwards-Sokal coupling of the Ising model and the random cluster measure. Firstly, we have that
\begin{align}\label{eq ES help}
\begin{split}
  \nu_{\beta,0}^{+} ( \eta &= -1 \text{ on } [-n,-1]    ) \\
&=    \mu_{\beta,0}^{+} ( \eta = -1 \text{ on } [-n,-1] \times \{0\} )
\\ &=   \bP_{\beta,0}^{+} ( \eta = -1 \text{ on } [-n,-1] \times \{0\}, \{[-n,-1] \leftrightarrow \infty \}^c) 
\\ &=  \bP_{\beta,0}^{+} ( \eta = +1 \text{ on } [-n,-1] \times \{0\}, \{[-n,-1] \leftrightarrow \infty\}^c ),
\end{split}
\end{align}
where the two first equalities follow merely by definition, and the last one follows since, in the Edwards-Sokal coupling, each finite cluster is colored $+1$ or $-1$ with equal probability, independently of the other clusters. On the other hand, we have that
\begin{align}
  \nu_{\beta,0}^{+} ( \eta &= +1 \text{ on } [-n,-1] ) \\ &=  \mu_{\beta,0}^{+} ( \eta = +1 \text{ on } [-n,-1] \times \{0\})
\\ &=  \bP_{\beta,0}^{+} ( \eta = +1 \text{ on } [-n,-1] \times \{0\}, \{[-n,-1] \leftrightarrow \infty \}) 
\\ &\quad+  \bP_{\beta,0}^{+} ( \eta = +1 \text{ on } [-n,-1] \times \{0\}, \{[-n,-1] \leftrightarrow \infty\}^c )
\\ &\geq  \mu_{\beta,0}^{+} ( \eta = +1 \text{ on } [-n,-1] \times \{0\} ) \hP_{\beta}^{+} (\{ [-n,-1] \leftrightarrow \infty \})
 \\ &\quad+  \mu_{\beta,0}^{+} ( \eta = -1 \text{ on } [-n,-1] \times \{0\} ),
\end{align}
where, in the last inequality, we used \eqref{eq ES help} and that 
\begin{align}
\bP_{\beta,0}^{+} ( \eta &= +1 \text{ on } [-n,-1] \times \{0\},\{ [-n,-1] \leftrightarrow \infty\} )
\\ \geq &\mu_{\beta,0}^{+} ( \eta = +1 \text{ on } [-n,-1] \times \{0\} ) \bP_{\beta,0}^{+} (\{ [-n,-1] \leftrightarrow \infty\} ).
\end{align}
The latter inequality holds since alignments in the Ising model increase the probability of bonds in the random cluster model, as seen by the basic Edwards-Sokal coupling (see  \citet[Corollary 6.4]{GeorgiiHaggstromMaes2001}). Hence, we obtain that,
\begin{align}
\nu_{\beta,0}^{+} ( \eta &= -1 \text{ on } [-n,-1] \times \{0\} ) \\ \leq  &\nu_{\beta,0}^{+} ( \eta = +1 \text{ on } [-n,-1] \times \{0\} )  \hP_{\beta}^{+} ( \{[-n,-1] \leftrightarrow \infty \}^c),
\end{align}
and thus, by Lemma \ref{lem exponential decay} we conclude that $\theta_{-} < \theta_{+}$, and hence also the proof of the proposition.
\end{proof}

\subsection{Proofs of Propositions \ref{prop unique}, \ref{prop not summable} and \ref{prop decimation}}\label{sec proofs prop}
The proofs of Propositions \ref{prop unique}, \ref{prop not summable} and \ref{prop decimation} all follow from already well established theory. For completeness, we present their arguments in this subsection.

\begin{proof}[Proof of Proposition \ref{prop unique}]
Consider the Ising model on $\bZ^2$ in the uniqueness regime.
\citet{SchonmannSchlosman1995} proved that this model is contained in the class of Gibbs models having complete analyticity as introduced in \citet{DobrushinShlosmanCAGF1985}. In particular, by Condition IIIc in \cite{DobrushinShlosmanCAGF1985}, the following holds: there are constants $C,c\in(0,\infty)$ such that for any $\Lambda \in \tilde{\cS}$ containing the origin and every $\sigma, \omega \in \tilde{\Omega}$ satisfying $\sigma_y = \omega_y$ for all $y \neq x$, $x \notin \Lambda$, we have that
\begin{align}\label{eq complete analyticity}
\Big|\mu_{\beta,h}^{\Lambda,\sigma}( \eta_{(0,0)}=+1) - \mu_{\beta,h}^{\Lambda,\omega}( \eta_{(0,0)}=+1)\Big|\leq Ce^{-c \norm{x}_1}.
\end{align}
Consider $\omega_{-\infty}^{-1} \in \Omega_{<0}$. For $n \in \bN$, let $\Lambda_n:= [-n,n]^2 \setminus ([-n,-1]\times\{0\})$ and set $\omega^{\pm} \in \tilde{\Omega}$ equal to $\omega_{-\infty}^{-1}$ on $[-n,-1]\times\{0\}$ and otherwise equal to $\pm$. Then, by the monotonicity of the Ising model, we have that 
\begin{align}\label{eq uniform bound DS}
\begin{split}
\nu_{\beta,h}^+(\eta_0 &=+ \mid +^{-(n+1)}_{-\infty} \omega^{-1}_{-n} )  - \nu_{\beta,h}^+(\eta_0=+ \mid -^{-(n+1)}_{-\infty} \omega^{-1}_{-n} )| 
\\ &\leq  \lim_{n \rightarrow \infty} |\mu_{\beta,h}^{\Lambda_n,\omega^+}(\eta_{(0,0)}=+1) - \mu_{\beta,h}^{\Lambda_n,\omega^-}(\eta_{(0,0)}=+1)| 
\\ &\leq \lim_{n \rightarrow \infty} 10n Ce^{-cn},
\end{split}
\end{align}
where the last inequality follows by \eqref{eq complete analyticity} and by telescoping along the boundary of $\Lambda_n$. Since the bound in \eqref{eq uniform bound DS} holds uniformly for all $\omega_{-\infty}^{-1} \in \Omega_{<0}$, it follows that the $g$-functions defined in \eqref{def particular $g$-functions} are equal. Thus, by Lemma \ref{lem uniqueness} b), we have that $\nu_{\beta,h}^+$ is the unique measure consistent with $P_{\beta,h}^{+,+}$ and that $P_{\beta,h}^{+,+}$ is regular. By this argument it also follows that $var_k(P_{\beta,h}^{+,+})$ decays exponentially in $k$ and is thus summable.
\end{proof} 

\begin{proof}[Proof of Proposition \ref{prop not summable}]
\cite{FernandezMaillardCCCandGibbs2004} gave sufficient conditions for a regular $g$-measure to be a Gibbs measure, and vice versa. For this, they considered a class of regular $g$-functions having a property which they named \emph{good future}, see \cite[Definition 4.5]{FernandezMaillardCCCandGibbs2004}. By \cite[Theorem 4.12]{FernandezMaillardCCCandGibbs2004}, $g$-measures consistent with a $g$-function having good future are Gibbs measures with respect to a continuous specification. Since $\nu_{\beta,0}^+$ is not a Gibbs measure whenever $\beta>\beta_c$, it thus follows that $\nu_{\beta,0}^+$ cannot be consistent with a $g$-function having good future. As noted in \cite[Remark 4.13]{FernandezMaillardCCCandGibbs2004}, this class contains the set of regular $g$-functions with summable variation. 
\end{proof}

\begin{proof}[Proof of Proposition \ref{prop decimation}]
The statement follows by a straightforward adaptation of \cite[Theorem 4.2]{MaesRedigMoffaertNGibbs1999}. 
 In particular, \cite{MaesRedigMoffaertNGibbs1999} proved that, for all $\beta$ large, and for $k\geq 4$, the measure $\nu_{1,k}$ is a Gibbs measure with exponentially decaying correlations. This statement can, by the exact the same argument, be extended to $\nu_{l,k}$ with $l >1$, by taking $k$ and $\beta$ sufficiently large. In particular, the measure satisfies the HUC condition of \cite{FernandezMaillardCCCandGibbs2004}. From this the statement of the proposition follows as a consequence of \cite[Theorem 4.16]{FernandezMaillardCCCandGibbs2004}.
 \end{proof}

\subsection{Proofs of Theorem \ref{thm Ising is g} and Theorem \ref{thm Schonmann $g$-measure}b)}
In the proof of Theorem \ref{thm Schonmann $g$-measure}a) we showed that
\begin{align}\label{eq local limit}
\lim_{n \rightarrow \infty} \mu_{\beta,h}^+( \cdot \mid \eta = -1 \text{ on } (-\infty,-n] \times\{0\} ) =  \mu_{\beta,h}^+(\cdot).
\end{align}
In this subsection we will strengthen this statement and prove Theorem \ref{thm Ising is g}. For this, the next lemma is of key importance. 

\begin{lem}\label{lem second AP}
For $n \in \nat$, let $A_n$ be the event that there exists an infinite nearest-neighbor path of $+$'s from $(n,0) \in \bZ^2$ which is entirely contained in $D_n:= \{ (x,y) \in \bZ^2 \colon  x\leq n \}$.  Then, for $\beta>\beta_c$ and for all $n$ large,
\begin{align}\label{eq agreement measure}
\mu_{\beta,0}^+(A_n  \mid \eta = -1 \text{ on } (-\infty,0) \times \{0\} ) >0.
\end{align} 
\end{lem}

\begin{proof}
Let $\beta>\beta_c$. Note that the measure on the left hand side of \eqref{eq agreement measure} is Markovian. Further we remark that, by  \eqref{eq local limit} and translation invariance, for all $n$ large, we have that
\begin{align}
&\mu_{\beta,0}^{+}\big( \eta(n,0)=1 \mid \eta = -1 \text{ on } (-\infty,0) \times \{0\} \big) > \mu_{\beta,0}^{-}\big(\eta(n,0)=1\big),
\end{align}
and that, by the strong FKG property, for any finite connected set $\Gamma$ containing $(n,0)$,
\begin{align} 
\mu_{\beta,0}^{+}\big(\eta(n,0)=1 \mid \eta = -1 \text{ on } \partial \Gamma\big) \leq \mu_{\beta,0}^{-}\big(\eta(n,0)=1\big),\end{align} 
where $\partial \Gamma$ is the outer boundary of $\Gamma$. In particular, all the assumptions of \cite[Theorem 8.1]{GeorgiiHaggstromMaes2001} are satisfied, from which we conclude that with positive probability there is an infinite  nearest-neighbor path of $+$'s which contains $(n,0)$. By the finite-energy property of the Ising model this immediately extends to any $x \in \bZ^2 \setminus  [(-\infty,0) \times \{0\}]$.

Next, we want to improve this and show that, with positive probability, there is such an infinite  nearest-neighbor path of $+$'s which contains $(n,0)$ and which is contained in $D_n$. For this, we need a slight modification of the proof of \cite[Proposition 8.3]{GeorgiiHaggstromMaes2001}, detailed as follows: Let $n$ be as above and note that also 
\[ \mu_{\beta,0}^+\big(\eta(n+1,0)=1  \mid \eta = -1 \text{ on } (-\infty,0) \times \{0\}\big) >\mu_{\beta,0}^{-}\big(\eta(n+1,0)=1\big).\]
Next, identify each $\xi \in \tilde{\Omega}$ with $(\xi(x),\xi(rx))_{x\in D_n} \in S^{D_n}$, where $S=\{-1,+1\}^2$ and $rx$ is the reflection of $x$ along the the border of $D_n$. 
Let $f\colon \tilde{\Omega} \rightarrow \bR$ be the function given by $f(\eta)=\eta(n,0) + \eta(n+1,0)$. Then, by our choice of $n$, we have that 
\begin{align}
\mu_{\beta,0}^+\big(f \mid \eta = -1 \text{ on }  (-\infty,0) \times \{0\} \big)>0.
\end{align} 
On the other hand, let $\Gamma \subset D_n$ be a finite set containing $(n,0)$ and $\tilde{\Gamma}= \Gamma \cup r\Gamma$. 
If $(\xi,\xi') \in S^{D_n}$ with $(\xi,\xi')\neq (+1,+1)$ on $\partial \Gamma \cap D_n$, then $\xi'(x) \leq - \xi(x)$ for each $\partial \Gamma \cap D_n$. 
Therefore, $(\xi,\xi') \leq (\xi,-\xi)$ (as elements of $\tilde{\Omega}$) on $\partial \tilde{\Gamma} = \partial_{D_n}\Gamma \cup r\partial_{D_n}\Gamma$. We can thus write, for any such $(\xi,\xi')$ satisfying that $\xi(x) = -1 \text{ on } \partial \Gamma \cap (-\infty,0) \times \{0\}$,
\begin{align}
\mu&_{\beta,0}^+\big(f \mid (\eta,\eta')=(\xi,\xi') \text{ on } \partial \Gamma \cap D_n, \xi(x) = -1 \text{ on } (-\infty,0) \times \{0\} \big) \qquad
\\ \leq &\mu_{\beta,0}^+\big(\eta(n,0) \mid  \eta = (\xi,\xi') \text{ on }{\tilde{\Gamma}} \big) 
 + \mu_{\beta,0}^+\big(\eta(n+1,0) \mid \eta = (\xi,\xi') \text{ on }{\tilde{\Gamma}}\big) 
\\ \leq & \mu_{\beta,0}^+\big(\eta(n,0) \mid  \eta = (\xi,-\xi) \text{ on }{\tilde{\Gamma}}\big) + \mu_{\beta,0}^+\big(\eta(n+1,0) \mid \eta = (\xi,-\xi) \text{ on }{\tilde{\Gamma}}\big),
\end{align}
which equals $0$ by the strong FKG property and by symmetry of the model. Consequently, all the assumptions of \cite[Theorem 8.1]{GeorgiiHaggstromMaes2001} are satisfied and hence, with positive probability, there exists an infinite nearest-neighbor path $\gamma$ contained in $D_n$ and starting from $(n,0)$ such that all spins along $\gamma$ and its reflection image $r\gamma$ are $+$. This implies the statement of the lemma by another application of the FKG property.
\end{proof}

\begin{lem}
Let $\beta>\beta_c$. The measure $\mu_{\beta,0}^{+}(\cdot \mid \eta= -1 \text{ on } (\infty,-1] \times \{0\})$ is trivial on the tail-$\sigma$-algebra.
\end{lem}
\begin{proof}
This follows by the same proof as the proof of \cite[Lemma 1]{SchonmannNGibbs1989}.
\end{proof}

As an immediate consequence of the two previous lemmas, we have the following corollary:

\begin{cor}\label{cor agreement}
Let $\beta>\beta_c$ and let $B_m := \cup_{n=1}^m A_n$, where $A_n$ is the event defined in Lemma \ref{lem second AP}. Then $\mu_{\beta,0}^+( B_n \mid \eta = -1 \text{ on } (-\infty,0)\times \{0\} ) \uparrow 1$ as $n \rightarrow \infty$.
\end{cor}

\begin{proof}[Proof of Theorem \ref{thm Ising is g}]
Recall the definitions of $C_{\theta}$ and $C_{\theta,n}$ given just before the statement of Theorem \ref{thm Ising is g}. Let $\beta>\beta_c$, and let $\hat{\theta} \in (0,\infty)$ be such that, for some constants $C>0$, we have
\begin{align}\label{eq convergence to plus estimates}
\Big| \mu_{\beta,0}^+ \big(\eta(0,0)=+1\big) -  \mu_{\beta,0}^{[-n,n]^2,+}\big(\eta(0,0)=+1\big)\Big| \leq Ce^{-\hat{\theta} n}.
\end{align}
The existence of such $\hat{\theta}$ is standard; see e.g.\ \cite[Theorem 3.62]{VelenikFriedliSM2017}. Next, let $A^{-1}_{-\infty} := \big\{  \eta = -1 \text{ on } (-\infty,0)\times \{0\}\big\}$ and $B^{-1}_{-\infty} :=\big\{  \eta = +1 \text{ on } (-\infty,0)\times \{0\} \big\}$ and note that, by the strong FKG property, 
for any $A\in \tilde{\cF}_{<0}$ we have
\begin{align}
\mu_{\beta,0}^+\big(\cdot \mid A^{-1}_{-\infty}\big) \leq \mu_{\beta,0}^+\big(\cdot \mid A\big) \leq \mu_{\beta,0}^+\big(\cdot \mid B^{-1}_{-\infty}\big),
\end{align}
where $\tilde{\cF}_{<0}$ is the $\sigma$-algebra on $\tilde{\Omega}$ concentrating on events on $(-\infty,0)\times \{0\}$. 
Therefore, it is sufficient for us to show that 
\[ \lim_{n \rightarrow \infty} \sup_{B \in \tilde{\cF}_{C_{\theta,n}}}\Big| \mu_{\beta,0}^+\big(B\mid A_{-\infty}^{-1}\big)-\mu_{\beta,0}^+\big(B \mid B_{-\infty}^{-1}\big)\Big|=0,\]
in order to prove the statement of the theorem. For this, we first note that
\begin{align}\label{eq help me help me}
\begin{split}
\big| \mu_{\beta,0}^+\big(&\cdot \mid A_{-\infty}^{-1}\big)-\mu_{\beta,0}^+\big(\cdot\mid B_{-\infty}^{-1}\big)\big|
\\ \leq &\mu_{\beta,0}^+\big(B_{n/2} \mid A_{-\infty}^{-1}\big) \big|\mu_{\beta,0}^+\big(\cdot \mid A_{-\infty}^{-1}, B_{n/2}\big) - \mu_{\beta,0}^+\big(\cdot \mid B_{-\infty}^{-1}\big)\big|
\\ + &\mu_{\beta,0}^+\big(B_{n/2}^c \mid A_{-\infty}^{-1}\big) \big|\mu_{\beta,0}^+\big(\cdot \mid A_{-\infty}^{-1}, B_{n/2}^c\big) - \mu_{\beta,0}^+\big(\cdot \mid B_{-\infty}^{-1}\big)\big|.
\end{split}
\end{align}
By Corollary \ref{cor agreement}, the second term of \eqref{eq help me help me} goes to $0$ as $n\rightarrow \infty$. Hence, we will focus on the first term: by the Markov property, we have that
\[ \mu_{\beta,0}^+(\cdot \mid B_{-\infty}^{-1}) \leq \mu_{\beta,0}^+(\cdot \mid A_{-\infty}^{-1}, B_{n/2}) \leq  \mu_{\beta,0}^+(\cdot \mid \eta = +1 \text{ on } \{n/2\} \times \bZ ).\]
Thus, by Strassen's Theorem, there exists a coupling $\hP_{0,1}$ of $\mu_{\beta,0}^+(\cdot \mid B_{-\infty}^{-1})$ and $\mu_{\beta,0}^+(\cdot\mid A_{-\infty}^{-1}, B_{n/2}) $ such that $\hP_{0,1} \left(\eta^1 \leq \eta^2 \right) =1$.
In particular, for any $B \in \tilde{\mathcal{F}}_{C_{\theta,n}}$, $n \in \bN$, and with $\theta < \hat{\theta}$, we have that
\begin{align}
\big| \mu_{\beta,0}^+&(B\mid B_{-\infty}^{-1}) - \mu_{\beta,0}^+(B\mid A_{-\infty}^{-1}, B_{n/2} ) \big| 
\\ \leq &\hP_{0,1} \left( \eta^1 \neq \eta^2 \text{ on } C_{\theta,n} \right)
 \\ \leq &\sum_{ x \in C_{\theta,n}} \hP_{0,1} \left( \eta^1(x) \neq \eta^2(x) \right) 
\\ =  &\sum_{ x \in C_{\theta,n}} \hP_{0,1} \left( \eta^1(x)=0, \eta^2(x)=1 \right) 
\\  =  &\sum_{ x \in C_{\theta,n}} \left( \hP_{0,1} \left( \eta^1(x)=0 \right) - \hP\left( \eta^2(x)=0 \right) \right)
\\  = &\sum_{ x \in C_{\theta,n}} \big| \mu_{\beta,0}^+\big(\eta(x)=1\mid A_{-\infty}^{-1}, B_{n/2} \big) - \mu_{\beta,0}^+\big(\eta(x)=1 \mid B_{-\infty}^{-1}\big) \big|
\\ \leq &\sum_{ x \in C_{\theta,n}} \big| \mu_{\beta,0}^+\big(\eta(x)=1\mid  \eta = 1 \text{ on } \{n/2\} \times \bZ\big) - \mu_{\beta,0}^+\big(\eta(x)=1\big) \big|.
\end{align}
Each term inside the latter sum is bounded by \eqref{eq convergence to plus estimates}. Thus, since $\theta <\hat{\theta}$, the sum goes to $0$ as $n \rightarrow \infty$ and this completes the proof of the theorem.
\end{proof}

\begin{proof}[Proof of Theorem  \ref{thm Schonmann $g$-measure}b)]
This is an immediate consequence of Theorem \ref{thm Ising is g}.
\end{proof}

\subsection*{Acknowledgement}
The authors would like to thank Aernout van Enter for stimulating discussions and valuable comments. S.A. Bethuelsen acknowledges the DFG, project GA582/7-2.


\begin{thebibliography}{36}
\providecommand{\natexlab}[1]{#1}
\providecommand{\url}[1]{\texttt{#1}}
\expandafter\ifx\csname urlstyle\endcsname\relax
  \providecommand{\doi}[1]{doi: #1}\else
  \providecommand{\doi}{doi: \begingroup \urlstyle{rm}\Url}\fi
  
  \bibitem[van~den Berg and Steif(1999)]{BergSteifFC1999}
J.~van~den Berg and J.~E. Steif.
\newblock On the existence and nonexistence of finitary codings for a class of
  random fields.
\newblock \emph{Ann. Probab.}, 27\penalty0 (3):\penalty0 1501--1522, 1999.

\bibitem[~Berghout~et~al.(2017)]{BerghoutFernandezVerbitskiy2017}
S.~Berghout, R.~Fernandez and E.~Verbitskiy.
\newblock On the relation between {G}ibbs and $g$-measures.
\newblock Preprint, 2017.
  
  \bibitem[{Bissacot} et~al.(2017{\natexlab{a}}){Bissacot}, {Ossami Endo}, {van
  Enter}, and {Le Ny}]{BissacotEndovanEnterLeNy17}
R.~{Bissacot}, E.~{Ossami Endo}, A.~C.~D. {van Enter}, and A.~{Le Ny}.
\newblock {Entropic repulsion and lack of the $g$-measure property for Dyson
  models}.
\newblock \emph{ArXiv e-prints:} \href{https://arxiv.org/abs/1705.03156}{1705.03156}


\bibitem[Chayes et~al.(1987)Chayes, Chayes, and
  Schonmann]{ChayesChayesSchonmann1987}
J.~T. Chayes, L.~Chayes, and R.~H. Schonmann.
\newblock Exponential decay of connectivities in the two-dimensional {I}sing
  model.
\newblock \emph{J. Statist. Phys.}, 49\penalty0 (3-4):\penalty0 433--445, 1987.

\bibitem[Dobrushin(1968)]{Dobrushin1968}
R.~L. Dobrushin.
\newblock Description of a random field by means of conditional probabilities
  and conditions for its regularity.
\newblock \emph{Teor. Verojatnost. i Primenen}, 13:\penalty0 201--229, 1968.

\bibitem[Dobrushin and Shlosman(1985)]{DobrushinShlosmanCAGF1985}
R.~L. Dobrushin and S.~B. Shlosman.
\newblock Completely analytical {G}ibbs fields.
\newblock 
  Birkh\"auser
  Boston, Boston, 
  1985.

\bibitem[Dobrushin and Shlosman(1997)]{DobrushinShlosman1997}
R.~L. Dobrushin and S.~B. Shlosman.
\newblock Gibbsian description on ``non-{G}ibbs'' fields.
\newblock \emph{Uspekhi Mat. Nauk}, 52\penalty0 (2(314)):\penalty0 45--58,
  1997.

\bibitem[van Enter et~al.(1993)van Enter, Fern\'andez, and
  Sokal]{EnterFernandezSokal1993}
A.\ C.~D. van Enter, R.\ Fern\'andez, and A.~D. Sokal.
\newblock Regularity properties and pathologies of position-space
  renormalization-group transformations: scope and limitations of {G}ibbsian
  theory.
\newblock \emph{J. Statist. Phys.}, 72\penalty0 (5-6):\penalty0 879--1167,
  1993.

\bibitem[Fern{\'a}ndez et~al.(2011)Fern{\'a}ndez, Gallo, and
  Maillard]{FernandezGalloMaillard2011}
R.\ Fern{\'a}ndez, S.\ Gallo, and G.\ Maillard.
\newblock Regular {$g$}-measures are not always {G}ibbsian.
\newblock \emph{Electron. Commun. Probab.}, 16:\penalty0 732--740, 2011.

\bibitem[Fern{\'a}ndez and Maillard(2004)]{FernandezMaillardCCCandGibbs2004}
R.\ Fern{\'a}ndez and G.\ Maillard.
\newblock Chains with complete connections and one-dimensional {G}ibbs
  measures.
\newblock \emph{Electron. J. Probab.}, 9:\penalty0 no. 6, 145--176, 2004.


\bibitem[Fern{\'a}ndez and Pfister(1997)]{FernandezPfisterNGibbs1997}
R.~Fern{\'a}ndez and C.-E. Pfister.
\newblock Global specifications and nonquasilocality of projections of {G}ibbs
  measures.
\newblock \emph{Ann. Probab.}, 25\penalty0 (3):\penalty0 1284--1315, 1997.


\bibitem[Friedli and Velenik(2017)]{VelenikFriedliSM2017}
S.~Friedli and Y.~Velenik.
\newblock \emph{Statistical mechanics of lattice systems: a concrete
  mathematical introduction}.
\newblock Cambridge University Press, 2017.

\bibitem[Gallesco et~al.(2018)Gallesco, Gallo, and Takahashi]{GALLESCO2018689}
C.\ Gallesco, S.\ Gallo, and D.~Y. Takahashi.
\newblock Dynamic uniqueness for stochastic chains with unbounded memory.
\newblock \emph{Stochastic Processes and their Applications}, 128\penalty0
  (2):\penalty0 689 -- 706, 2018.

\bibitem[Gallo and Paccaut(2013)]{GalloPaccaut2013}
S.\ Gallo and F.\ Paccaut.
\newblock On non-regular {$g$}-measures.
\newblock \emph{Nonlinearity}, 26\penalty0 (3):\penalty0 763--776, 2013.

\bibitem[Gallo and Takahashi(2014)]{GalloTakahaskiAttractive2014}
S.\ Gallo and D.~Y. Takahashi.
\newblock Attractive regular stochastic chains: perfect simulation and phase
  transition.
\newblock \emph{Ergodic Theory Dynam. Systems}, 34\penalty0 (5):\penalty0
  1567--1586, 2014.

\bibitem[Georgii(2011)]{Georgii2011}
H.-O.\ Georgii.
\newblock \emph{Gibbs measures and phase transitions}, 
\newblock Walter de Gruyter \& Co., Berlin, second edition, 2011.

\bibitem[Georgii et~al.(2001)Georgii, H{\"a}ggstr{\"o}m, and
  Maes]{GeorgiiHaggstromMaes2001}
H.-O.\ Georgii, O.\ H{\"a}ggstr{\"o}m, and C.\ Maes.
\newblock The random geometry of equilibrium phases.
\newblock In \emph{Phase transitions and critical phenomena, {V}ol. 18},
  volume~18 of \emph{Phase Transit. Crit. Phenom.}, pages 1--142. Academic
  Press, San Diego, CA, 2001.

\bibitem[Harris(1955)]{harris1955}
T.~E. Harris.
\newblock On chains of infinite order.
\newblock \emph{Pacific J. Math.}, 5:\penalty0 707--724, 1955.

\bibitem[Hulse(1991)]{HulseAttractive1991}
P.\ Hulse.
\newblock Uniqueness and ergodic properties of attractive {$g$}-measures.
\newblock \emph{Ergodic Theory Dynam. Systems}, 11\penalty0 (1):\penalty0
  65--77, 1991.


\bibitem[Kalikow(1990)]{Kalikow1990}
S.\ Kalikow.
\newblock Random {M}arkov processes and uniform martingales.
\newblock \emph{Israel Journal of Mathematics}, 71\penalty0 (1):\penalty0
  33--54, 1990.

\bibitem[Keane(1972)]{Keane1972}
M.\ Keane.
\newblock Strongly mixing $g$-measures.
\newblock \emph{Inventiones mathematicae}, 16\penalty0 (4):\penalty0 309--324,
  1972.

\bibitem[Lanford and Ruelle(1969)]{LanfordRuelle69}
O.~E. Lanford and D.~Ruelle.
\newblock Observables at infinity and states with short range correlations in
  statistical mechanics.
\newblock \emph{Comm. Math. Phys.}, 13:\penalty0 194--215, 1969.

\bibitem[Liggett and Steif(2006)]{LiggettSteifSD2006}
T.~M. Liggett and J.~E. Steif.
\newblock Stochastic domination: the contact process, {I}sing models and {FKG}
  measures.
\newblock \emph{Ann. Inst. H. Poincar\'e Probab. Statist.}, 42\penalty0
  (2):\penalty0 223--243, 2006.

\bibitem[L\"orinczi(1995)]{LorincziProjectedGibbs1995}
J.\ L\"orinczi.
\newblock Quasilocality of projected {G}ibbs measures through analyticity
  techniques.
\newblock \emph{Helv. Phys. Acta}, 68\penalty0 (7-8):\penalty0 605--626, 1995.

\bibitem[L\"orinczi and Vande~Velde(1994)]{LorincziVelde1994}
J.\ L\"orinczi and K.\ Vande~Velde.
\newblock A note on the projection of {G}ibbs measures.
\newblock \emph{J. Statist. Phys.}, 77\penalty0 (3-4):\penalty0 881--887, 1994.

\bibitem[Maes et~al.(1999)Maes, Redig, and
  Van~Moffaert]{MaesRedigMoffaertNGibbs1999}
C.~Maes, F.~Redig, and A.~Van~Moffaert.
\newblock The restriction of the {I}sing model to a layer.
\newblock \emph{J. Statist. Phys.}, 96\penalty0 (1-2):\penalty0 69--107, 1999.

\bibitem[Maes et~al.(1999)Maes, Redig, and
  Van~Moffaert]{MaesRedigvonMoffaert99}
C.~Maes, F.~Redig, and A.~Van~Moffaert.
\newblock Almost Gibbsian versus weakly Gibbsian measures
\newblock \emph{Stochastic Processes and their Applications}, 79\penalty0 (1):\penalty0 1--15, 1999.  

\bibitem[Maes and Vande~Velde(1997)]{MaesVandeVelde1997}
C.\ Maes and K.\ Vande~Velde.
\newblock Relative energies for non-{G}ibbsian states.
\newblock \emph{Comm. Math. Phys.}, 189\penalty0 (2):\penalty0 277--286, 1997.

\bibitem[Onicescu and Mihoc(1935)]{OnicescuMihoc35}
O.~Onicescu and G.~Mihoc.
\newblock Sur les cha\^ines statistiques.
\newblock \emph{C. R. Acad. Sci. Paris,}, 200:\penalty0 511-512, 1935.

\bibitem[Pfister and Velenik(1998)]{PfisterVelenik1998}
C.\ E.\ Pfister and Y. Velenik.
\newblock Macroscopic description of phase separation in the 2{D}
              {I}sing model.
\newblock World Sci. Publ., River Edge, NJ, 1998.

\bibitem[Samson(200)]{Samson2000} P.-M.\ Samson.
\newblock Concentration of measure inequalities for {M}arkov chains and
              {$\Phi$}-mixing processes.
\newblock \emph{Ann. Probab.}, 28\penalty0 (1):\penalty0 416--461, 2000.


\bibitem[Schonmann(1987)]{SchonmannLDPising1987}
R.~H. Schonmann.
\newblock Second order large deviation estimates for ferromagnetic systems in
  the phase coexistence region.
\newblock \emph{Comm. Math. Phys.}, 112\penalty0 (3):\penalty0 409--422, 1987.

\bibitem[Schonmann(1989)]{SchonmannNGibbs1989}
R.~H. Schonmann.
\newblock Projections of {G}ibbs measures may be non-{G}ibbsian.
\newblock \emph{Comm. Math. Phys.}, 124\penalty0 (1):\penalty0 1--7, 1989.

\bibitem[Schonmann and Shlosman(1995)]{SchonmannSchlosman1995}
R.~H. Schonmann and S.~B. Shlosman.
\newblock Complete analyticity for {$2$}{D} {I}sing completed.
\newblock \emph{Comm. Math. Phys.}, 170\penalty0 (2):\penalty0 453--482, 1995.



\bibitem[Warfheimer(2010)]{WarfheimerSDFuzzy2010}
M.\ Warfheimer.
\newblock Stochastic domination for the {I}sing and fuzzy {P}otts models.
\newblock \emph{Electron. J. Probab.}, 15:\penalty0 no. 58, 1802--1824, 2010.
  
\end{thebibliography}
\end{document}